\documentclass[11pt]{article}
\usepackage{amssymb}
\usepackage[perpage,symbol]{footmisc}
\usepackage{listings}
\usepackage{amsfonts}
\usepackage{enumerate}
\usepackage{amsmath}
\usepackage{cite}
\usepackage{array}
\usepackage{booktabs}
\usepackage{pgf,tikz}
\usetikzlibrary{calc,arrows}
\usepackage{latexsym,bm}
\usepackage{graphicx}
\usepackage{color}
\usepackage{amscd}
\usepackage{epsfig}
\usepackage{circuitikz}

\begin{document}

\newtheorem{theorem}{Theorem}[section]
\newtheorem{corollary}[theorem]{Corollary}
\newtheorem{definition}[theorem]{Definition}
\newtheorem{proposition}[theorem]{Proposition}
\newtheorem{lemma}[theorem]{Lemma}
\newtheorem{example}[theorem]{Example}
\newtheorem{algorithm}[theorem]{Algorithm}
\newenvironment{proof}{\noindent {\bf Proof.}}{\rule{3mm}{3mm}\par\medskip}
\newcommand{\remark}{\medskip\par\noindent {\bf Remark.~~}}
\title{Multiplication of polynomials over the binary field}
\author{Chunlei Liu\footnote{Shanghai Jiao Tong Univ., Shanghai 200240, clliu@sjtu.edu.cn.}}
\date{}
\maketitle
\thispagestyle{empty}

\abstract{Additive Fourier Transform for polynomials over the binary field is studied. The technique of Gao-Mateer is generalized, giving an algorithm which compute the product of two polynomials of degree $<n$ over the binary field in
 $O(n\log n(\log\log n)^2)$ bit operations.}


\section{\small{INTRODUCTION}}
\hskip .2in
Fast multiplication of polynomials is an important problem in computational number theory, computer algebra, algebraic coding theory and in cryptography.
In fast, fast multiplication of polynomials gives rise to fast division of polynomials, and thus gives rise to fast Generalized Euclidean Algorithm of polynomials \cite{AHU74,Kn81}. Similarly, fast Midway Euclidean Algorithm of polynomial, which plays an equal role as Berlekamp-Massey Algorithm \cite{Be68, Ma69, Mc78}.

Researches on Fast Polynomial Multiplication have given birth to  ingenious ideas as well as beautiful results. Contributors to these ideas are Karatsuba-Ofman \cite{KO63}, Cooley-Tukey \cite{CT65}, Wang-Zhu\cite{WZ88}, Gathen-Gerhard \cite{GG96},  Sch\"{o}nhage \cite{Sch77}, Cantor \cite{Ca89}, and etc.

In 1965, Cooley-Tukey invented the Fast Fourier Transform to get a  Fast Polynomial Multiplication Algorithm. In 1977, Sch\"{o}nhage invented the Number Theoretic Transform to get a  Fast Polynomial Multiplication Algorithm.
And, Wang-Zhu and Cantor invented the Additive Fourier Transform to get a  fast polynomial multiplication algorithm.

In the Additive Fourier Transform approach to fast polynomial multiplication, a fast evaluation algorithm is needed. In 1996, Gao-Matter \cite{GM10} introduced an evaluation algorithm, which is efficient for polynomials whose degree is of the form $2^{2^l}$.
We shall generalize the method of Gao-Mateer to prove the following theorem.
\begin{theorem}\label{main}There exists a polynomial multiplication  algorithm  which computes the product of two polynomials of degree $<n$ over ${\rm GF}(2)$ in
 $O(n\log n(\log\log n)^2)$ bit operations.\end{theorem}
Notice that, Wang-Zhu's algorithm may yield the bit complexity
$O(n(\log_2 n)^2)$,  Cantor's algorithm yields the bit complexity
$O(n(\log_2 n)^{\log_23/2})$,  and Gao-Mateer's algorithm may yields the bit complexity
$O(n(\log_2 n)(\log\log n)^2)$ when $n$ is of the form $2^{2^l}$.
\section{Evaluation and Remaindering}
\paragraph{}In this section we generalize the evaluation algorithm of Gao-Mateer \cite{GM10}. We shall split it into an evaluation algorithm ${\rm EA}(f,a,m,L)$ and a remaindering algorithm ${\rm RA}(f,b,m,l,L)$.
We begin with the following definition.
\begin{definition}
Let $L$ be a positive integer. A basis $(\beta_1,\cdots,\beta_{2^L})$ of ${\rm GF}(2^{2^L})$ over ${\rm GF}(2)$ is called a Cantor basis if
$\beta_1=1$, and $\beta_i^2-\beta_i=\beta_{i-1}$ for all $i>1$.\end{definition}
Let $(\beta_1,\cdots,\beta_{2^L})$ be a Cantor basis of ${\rm GF}(2^{2^L})$ over ${\rm GF}(2)$. We write
$$W_i=\sum_{j=1}^i{\rm GF}(2)\beta_j,\ i=1,\cdots,2^L.$$
One can show that $W_i$ is independent of the choice of the Cantor basis.
The basis $(\beta_1,\cdots,\beta_{2^L})$ certainly gives an isomorphism from ${\rm GF}(2)^{2^L}$ to ${\rm GF}(2^{2^L})$, and thus the lexicographic order in  ${\rm GF}(2)^{2^L}$ induces an order in ${\rm GF}(2^{2^L})$.
Under this induced order, elements of ${\rm GF}(2^{2^L})$ can be ordered increasingly as $\varpi_0,\varpi_1,\cdots,\varpi_{2^{2^L}-1}$. That is,
$$\varpi_0<\varpi_1<\cdots<\varpi_{2^{2^L}-1}.$$
The tasks of  ${\rm EA}(f,a,m,L)$ and ${\rm RA}(f,b,m,l,L)$ are described in the following tables.
\vskip0.2in
\begin{tabular}{ll}
  &Task of ${\rm EA}(f,a,m,L)$\\
\hline
 input 1  & $L$: a number;\\
 input 2  & $m$: a number  $\leq 2^L$;\\
input 3& $a$: an element in $W_{2^L}\backslash W_m$;\\
input 4&$f$: a polynomial over ${\rm GF}(2^{2^L})$ of degree $<2^m$;\\
install & $(\beta_1,\cdots,\beta_{2^L})$: a Cantor basis of ${\rm GF}(2^{2^L})$ over ${\rm GF}(2)$;\\
   output& $(v_0,\cdots,v_{2^m-1})$: $v_i=f(a+\varpi_i)$, $i=0,\cdots,2^m-1$;\\
\hline
\end{tabular}
\vskip0.2in

\vskip0.2in
\begin{tabular}{ll}
&Task of ${\rm RA}(f,b,m,l,L)$\\
\hline
 input 1  & $L$: a number;\\
 input 2  & $l$: a number  $\leq L$;\\
 input 3  & $m$: a number  satisfying $2^{l-1}<m\leq 2^l$;\\
input 4& $b$: an element in $W_{2^L}\backslash W_{m-2^{l-1}}$;\\
input 5&$f$: a polynomial over ${\rm GF}(2^{2^L})$ of degree $<2^m$;\\
install & $(\beta_1,\cdots,\beta_{2^L})$: a Cantor basis of ${\rm GF}(2^{2^L})$ over ${\rm GF}(2)$;\\
   output& $(r_0,\cdots,r_{2^{m-2^{l-1}}-1})$: $r_j$ is the remainder of $f$\\
& modulo $x^{2^{2^{l-1}}}-x-(b+\varpi_j)$;\\
\hline
\end{tabular}
\vskip0.2in

We now describe the mechanisms of the evaluation algorithm and the remaindering algorithm.
\begin{lemma}[evaluation mechanism]
Let $L,m$ be numbers with  $1<m\leq 2^L$, and $l=\lceil\log_2m\rceil$. Let $a\in W_{2^L}\backslash W_m$, and $b=a^{2^{2^{l-1}}}-a$.
Let $f$ be a polynomial over ${\rm GF}(2^{2^L})$ of degree $<2^m$.
Let $0\leq i<2^{l-1}, 0\leq j<2^{m-2^{l-1}}$,  $a_j=a+\varpi_{j2^{2^{l-1}}}$,
and $r_j$ the remainder of $f$ modulo $x^{2^{2^{l-1}}}-x-(b+\varpi_j)$.
Then
$$f(a+\varpi_{i+j2^{2^{l-1}}})=r_j(a_j+\varpi_i).$$
\end{lemma}
\begin{proof}
It is easy to see that
$$(a+\varpi_{j2^{2^{l-1}}}+\varpi_i)^{2^{2^{l-1}}}
-(a+\varpi_{j2^{2^{l-1}}}+\varpi_i)=b+\varpi_j.$$
That is, $x^{2^{2^{l-1}}}-x-(b+\varpi_j)$ vanishes at $a+\varpi_{j2^{2^{l-1}}}+\varpi_u$.
As $r_j$ is the remainder of $f$ modulo $x^{2^{2^{l-1}}}-x-(b+\varpi_j)$,
we have
$$f(a+\varpi_{j2^{2^{l-1}}}+\varpi_i)=r_j(a+\varpi_{j2^{2^{l-1}}}+\varpi_i)
=r_j(a_j+\varpi_i).$$
It is easy to see that
$\varpi_{j2^{2^{l-1}}}+\varpi_i=\varpi_{i+j2^{2^{l-1}}}$.
Therefore $f(a+\varpi_{i+j2^{2^{l-1}}})=r_j(a_j+\varpi_i)$.
The lemma is proved.
\end{proof}
The above lemma reduces the evaluation of $f$ to simultaneous remaindering of $f$ modulo $x^{2^{2^{l-1}}}-x-(b+\varpi_j)$, $j=0,\cdots, 2^{m-2^{l-1}}-1.$
Before describing the remaindering mechanism, we recall following expansion of polynomials introduced by Gao-Mateer \cite{GM10}.
\begin{definition} Let $k$ be a positive number, and $f(x)$ a polynomial over $\overline{{\rm GF}(2)}$.  The Gao-Mateer polynomials of $f$ at $x^{2^k}-x$
is the tuple  $(\hat{f}_0,\cdots,\hat{f}_{2^k-1})$ of polynomials over $\overline{{\rm GF}(2)}$ such that
$$f(x)=\sum_{i=0}^{2^k-1}x^i\hat{f}_i(x^{2^k}-x)^j.$$\end{definition}
\begin{lemma}[remaindering mechanism]
Let $L,m$ be numbers with  $1<m\leq 2^L$. Let $l=\lceil\log_2m\rceil$, and $b\in W_{2^L}\backslash W_{m-2^{l-1}}$.
Let $f$ be a polynomial over ${\rm GF}(2^{2^L})$ of degree $<2^m$, and $(\hat{f}_0,\cdots,\hat{f}_{2^{2^{l-1}}-1})$ the Gao-Mateer polynomials of $f$ at $x^{2^{2^{l-1}}}-x$.
Let $0\leq i<2^{l-1}, 0\leq j<2^{m-2^{l-1}}$,
$r_{ij}=\hat{f}_i(b+\varpi_j)$, and
$r_j(x)=\sum_{i=0}^{2^{2^{l-1}}-1}r_{ij}x^i$.
Then
 $r_j$ is the remainder of $f$ modulo $x^{2^{2^{l-1}}}-x-(b+\varpi_j)$.
\end{lemma}
\begin{proof}
In fact,
\begin{eqnarray*}
f(x)
&=&\sum_{i=0}^{2^{2^{l-1}}-1}x^i\hat{f}_i(x^{2^{2^{l-1}}}-x)\\
&\equiv&\sum_{i=0}^{2^{2^{l-1}}-1}x^i\hat{f}_i(b+\varpi_j)\\
&\equiv&\sum_{i=0}^{2^{2^{l-1}}-1}r_{ij}x^i
\\
&\equiv&r_j(x)\left({\rm mod }x^{2^{2^{l-1}}}-x-(b+\varpi_j)\right).
\end{eqnarray*}
The lemma is proved.
\end{proof}
The above lemma reduces to remaindering of $f$ to evaluations of
$\hat{f}_0,\cdots,\hat{f}_{2^{2^{l-1}}-1}$. The above two lemmas together show that the evaluation algorithm and the remaindering algorithm interplay between each other, and  justify the following pseudo-codes.
\vskip0.2in
\begin{tabular}{ll}
  &Pseudo-codes of ${\rm EA}(f,a,m,L)$\\
\hline
   1& if $m=1$; \\
2& $(v_0,v_1):=(f(a),f(a+1))$;\\
3&else \\
4&$l:=\lceil\log_2m\rceil$;\\
5&$b:=a^{2^{2^{l-1}}}-a$;\\
6&  $r:={\rm RA}(f,b,m,l,L)$;\\
7&for $j=0$ to $2^{m-2^{l-1}}-1$;\\
8&  $a_j:=a+\varpi_{j2^{2^{l-1}}}$;\\
9&$(v_{0+j2^{2^{l-1}}},\cdots,v_{2^{2^{l-1}}-1+j2^{2^{l-1}}}):={\rm EA}(r_j,a_j,2^{l-1},L)$;\\
10&end for\\
11&end if\\
&Return $(v_0,\cdots,v_{2^m-1})$;\\
\hline
\end{tabular}

\vskip0.4in
\begin{tabular}{ll}
&Pseudo-codes of ${\rm RA}(f,b,m,l,L)$\\
\hline
   1&$\hat{f}:=$ the Gao-Mateer polynomials of $f$ at $x^{2^{2^{l-1}}}-x$;\\
2& for $i=0$ to $2^{2^{l-1}}-1$;\\
3&$(r_{i0},\cdots,r_{i,2^{m-2^{l-1}}-1}):={\rm EA}(\hat{f}_i,b,m-2^{l-1},L)$;\\
 4&end for\\
5& for $j=0$ to $2^{m-2^{l-1}}-1$;\\
6&$r_j:=\sum_{i=0}^{2^{m-2^{l-1}}-1}r_{ij}x^i$;\\
7&end for\\
8&Return $(r_0,\cdots,r_{2^{m-2^{l-1}}})$;\\
\hline
\end{tabular}
\begin{theorem}\label{multiplication}
The evaluation algorithm ${\rm EA}(f,a,m,L)$ costs at most $2^{m-1}m$ multiplications in ${\rm GF}(2^{2^L})$.
\end{theorem}
\begin{proof}
Let $E_0(m)$ and $R_0(m)$ be the maximum number of multiplications in running ${\rm EA}(f,a,m,L)$ and ${\rm RA}(f,b,m,l,L)$ respectively.
If $m=1$, then  by Line 2 in ${\rm EA}(f,a,m,L)$,
$$E_0(m)=2^{m-1}m.$$
Assume that $m>1$.
By Lines 6 and 9 in ${\rm EA}(f,a,m,L)$, $$E_0(m)\leq2^{m-2^{l-1}}E_0(2^{l-1})+R_0(m).$$
By Gao-Mateer's result, 
 Line 1  in ${\rm RA}(f,b,m,l,L)$
costs no multiplications.
So, by Line 3 in ${\rm RA}(f,b,m,l,L)$,
$$R_0(m)\leq 2^{2^{l-1}}E_0(m-2^{l-1}).$$
Thus $$E_0(m)\leq 2^{m-2^{l-1}}E_0(2^{l-1})+2^{2^{l-1}}E_0(m-2^{l-1}).$$
Hence, by induction, $$E_0(m)\leq2^{m-2^{l-1}}\cdot 2^{2^{l-1}-1}2^{l-1}+2^{2^{l-1}}\cdot2^{m-2^{l-1}-1} (m-2^{l-1})=2^{m-1}m.$$
The theorem is proved.
\end{proof}
\begin{lemma}
Let $E_1(m)$ be the maximum number of additions in running the algorithm  ${\rm EA}(f,a,m,L)$. If $m>1$, then
$$\frac{E_1(m)}{2^m}\leq \frac{E_1(2^{l-1})}{2^{2^{l-1}}}
+ \frac{E_1(m-2^{l-1})}{2^{m-2^{l-1}}}+m-2^{l-1}, \ l=\lceil\log_2m\rceil.$$
\end{lemma}
\begin{proof}
Let  $R_1(m)$ be the maximum number of additions in  running the algorithm ${\rm RA}(f,b,m,l,L)$.
Then
by Lines 6 and 9 in ${\rm EA}(f,a,m,L)$,  $$E_1(m)\leq2^{m-2^{l-1}}E_1(2^{l-1})+R_1(m).$$
By  Gao-Mateer's result, 
 Line 1  in ${\rm RA}(f,b,m,l,L)$
costs  $2^m(m-2^{l-1})$ additions. So
 by Line 3 in ${\rm RA}(f,b,m,l,L)$,
$$R_1(m)\leq 2^{2^{l-1}}E_1(m-2^{l-1})+2^m(m-2^{l-1}).$$
Thus $$E_1(m)\leq 2^{m-2^{l-1}}E_1(2^{l-1})+2^{2^{l-1}}E_1(m-2^{l-1})+2^{m-1}(m-2^{l-1}).$$ That is,
$$\frac{E_1(m)}{2^m}\leq \frac{E_1(2^{l-1})}{2^{2^{l-1}}}
+ \frac{E_1(m-2^{l-1})}{2^{m-2^{l-1}}}+m-2^{l-1}.$$
The lemma is proved.
\end{proof}

\begin{corollary}[Gao-Mateer]
The algorithm ${\rm EA}(f,b,2^l,L)$ costs at most $2^{2^l}2^l(1+\frac{l}{2})$ additions.
\end{corollary}
\begin{proof}
Let $E_1(m)$ be the maximum number of additions in running ${\rm EA}(f,a,m,L)$.
By Line 2 of ${\rm EA}(f,a,m,L)$, the corollary is true if $l=0$.
By the last lemma,
$$\frac{E_1(2^l)}{2^{2^l}}\leq \frac{2E_1(2^{l-1})}{2^{2^{l-1}}}
+\frac{2^{l-1}}{2}.$$
Thus, by induction on $l$,
$$\frac{E_1(2^l)}{2^{2^l}}\leq 2^l(1+\frac{l-1}{2})
+2^{l-1}\leq 2^l(1+\frac{l}{2}).$$
The corollary is proved.
\end{proof}

\begin{theorem}\label{addition}
The algorithm ${\rm EA}(f,a,m,L)$ costs no more than $2^mm(1+\log_2 m)$ additions.
\end{theorem}
\begin{proof}
By Line 2 of ${\rm EA}(f,a,m,L)$, we may assume that $m>1$.
Let $E_1(m)$ be the maximum number of additions in running ${\rm EA}(f,a,m,L)$. Write $m=\sum_{i=1}^t2^{d_i}$ with $\{d_1,\cdots,d_t\}\subseteq\{0,1,\cdots,l-1\}$.
By induction, the last lemma implies that
$$\frac{E_1(m)}{2^m}\leq \sum_{i=1}^t\frac{E_1(2^{d_i})}{2^{2^{d_i}}}
+\sum_{i=1}^{t-1}2^{d_i}(t-i).$$
By the last corollary, we have
\begin{eqnarray*}
\frac{E_1(m)}{2^m}&\leq& \sum_{i=1}^t2^{d_i}(1+\frac{d_i}{2})
+\sum_{i=1}^{t-1}2^{d_i}(t-i)\\
&\leq& (1+\frac{\log_2m}{2})\sum_{i=1}^t2^{d_i}
+\frac{\log_2m}{2}\sum_{i=1}^t2^{d_i}\\
&\leq& (1+\log_2m)m.
\end{eqnarray*}
Thus $E_1(m)\leq 2^m m(1+\log_2m)$.
The theorem is proved.
\end{proof}
\section{Interpolation and Dividend Reconstruction}
\hskip 0.2in  In this section, we describe an interpolation algorithm ${\rm IA}(f,a,m,L)$ and a dividend reconstruction algorithm ${\rm DRA}(r,b,m,l,L)$.

The tasks of ${\rm IA}(v,a,m,L)$ and ${\rm DRA}(r,b,m,l,L)$ are described in the following tables.
\vskip0.2in
\begin{tabular}{ll}
  &Task of  ${\rm IA}(v,a,m,L)$\\
\hline
 input 1  & $L$: a number;\\
 input 2  & $m$: a number  $\leq 2^L$;\\
input 3& $a$: an element in $W_{2^L}\backslash W_m$;\\
input 4&$v=(v_0,\cdots,v_{2^m-1})$: a tuple over ${\rm GF}(2^{2^L})$;\\
install & $(\beta_1,\cdots,\beta_{2^L})$: a Cantor basis of ${\rm GF}(2^{2^L})$ over ${\rm GF}(2)$;\\
   output& $f$: a polynomial over ${\rm GF}(2^{2^L})$ of degree $<2^m$ such that \\
 & $f(a+\varpi_i)=v_i$, $i=0,\cdots,2^m-1$;\\
  \hline
\end{tabular}
\vskip0.2in
\vskip0.2in
\begin{tabular}{ll}
&Task of ${\rm DRA}(r,b,m,l,L)$\\
\hline
input 1& $l$: a positive number $\leq L$;\\
input 2& $m$: a number satisfying $2^{l-1}<m\leq 2^l$;\\
input 3& $b$: a number $<2^{L-m}$;\\
input 4& $r=(r_0,\cdots,r_{2^{m-2^{l-1}}})$: polynomials over ${\rm GF}(2^{2^L})$ of degree $<2^{2^{l-1}}$;\\
install & $(\beta_1,\cdots,\beta_{2^L})$: a Cantor basis of ${\rm GF}(2^{2^L})$ over ${\rm GF}(2)$;\\
output&$f$: a polynomial over ${\rm GF}(2^{2^L})$ of degree $<2^m$ whose remainder\\
& modulo $x^{2^{2^{l-1}}}-x-(b+\varpi_j)$ is $r_j$;\\
\hline
\end{tabular}
\vskip0.2in
We now describe the mechanisms of the interpolation algorithm and the dividend reconstruction algorithm.
\begin{lemma}[interpolation mechanism]
Let $L,m$ be numbers with  $1<m\leq 2^L$, and $l=\lceil\log_2m\rceil$. Let $a\in W_{2^L}\backslash W_m$, and $b=a^{2^{2^{l-1}}}-a$.
Let $(v_0,\cdots,v_{2^m-1})$ be a tuple over ${\rm GF}(2^{2^L})$.
Let $0\leq i<2^{l-1}, 0\leq j<2^{m-2^{l-1}}$,  $a_j=a+\varpi_{j2^{2^{l-1}}}$,
and $r_j(x)$ the polynomial  over ${\rm GF}(2^{2^L})$ of degree $<2^{2^{l-1}}$ such that
$$r_j(a_j+\varpi_i)=v_{i+j2^{2^{l-1}}}.$$
If $f(x)$ is the polynomial  over ${\rm GF}(2^{2^L})$ of degree $<2^m$ whose remainders to the moduli  $$x^{2^{2^{l-1}}}-x-(b+\varpi_j),\ j=0,\cdots, 2^{m-2^{l-1}}-1$$ are $r_0(x),\cdots,r_{2^{m-2^{l-1}}-1}$.
Then
$$f(a+\varpi_u)=v_u,\ u=0,\cdots,2^m-1.$$
\end{lemma}
\begin{proof}
This follows from the evaluation mechanism.
\end{proof}
The above lemma reduces the construction of an interpolation polynomial to the reconstruction of a polynomial whose remainders to the moduli $$x^{2^{2^{l-1}}}-x-(b+\varpi_j),\ j=0,\cdots, 2^{m-2^{l-1}}-1$$ are given.
\begin{lemma}[dividend reconstruction mechanism]
Let $L,m$ be numbers with  $1<m\leq 2^L$. Let $l=\lceil\log_2m\rceil$, and $b\in W_{2^L}\backslash W_{m-2^{l-1}}$.
Let $r_0,\cdots,r_{2^{m-2^{l-1}}-1}$ be polynomials over ${\rm GF}(2^{2^L})$ of degree $<2^{2^{l-1}}$. Write
$$r_j(x)=\sum_{i=0}^{2^{2^{l-1}}-1}r_{ij}x^i,\ 0\leq j<2^{m-2^{l-1}}.$$
For $0\leq i<2^{l-1}$, let $t_i$ be the polynomial over ${\rm GF}(2^{2^L})$ of degree $<2^{2^{l-1}}$.
$$t_i(b+\varpi_j)=r_{ij}, \ 0\leq j<2^{m-2^{l-1}}.$$
If $f$ is the polynomial whose Gao-Mateer polynomials at $x^{2^{2^{l-1}}}-x$ are $(t_0,\cdots,t_{2^{2^{l-1}}-1})$,
Then the remainder of $f$ modulo $x^{2^{2^{l-1}}}-x-(b+\varpi_j)$
is $r_j$ for $j=0,\cdots,2^{m-2^{l-1}}-1$.
\end{lemma}
\begin{proof}
This follows from the remaindering mechanism.
\end{proof}
The above lemma reduces to the dividend reconstruction to construction of interpolation polynomials. The above two lemmas together show that the interpolation algorithm and the dividend recinstruction algorithm interplay between each other, and  justify the following pseudo-codes.
\vskip0.2in
\begin{tabular}{ll}
  &Pseudo-codes of  ${\rm IA}(v,a,m,L)$\\
\hline
   1& if  $m=1$; \\
2& $f(x):=v_0-a(v_0+v_1)+(v_0+v_1)x$;\\
3&else \\
4&$l:=\lceil\log_2m\rceil$;\\
5&for $j=0$ to $2^{m-2^{l-1}}-1$;\\
6&  $a_j:=a+\varpi_{j2^{2^{l-1}}}$;\\
7&$w_j:=(v_{0+j2^{2^{l-1}}},\cdots,v_{2^{2^{l-1}}-1+j2^{2^{l-1}}})$;\\
8&$r_j:={\rm IA}(w_j,a_j,2^{l-1})$;\\
9&end for\\
10&$r=(r_0,\cdots, r_{2^{m-2^{l-1}}-1})$;\\
11& $b:=a^{2^{2^{l-1}}}-a$;\\
12& $f:={\rm DRA}(r,b,m,l,L)$;\\
13&end if\\
&Return $f$;\\
\hline
\end{tabular}
\vskip0.2in
\vskip0.2in
\begin{tabular}{ll}
&Pseudo-codes of ${\rm DRA}(r,b,m,l,L)$\\
\hline
1& for $j=0$ to $2^{m-2^{l-1}}-1$;\\
2&$(r_{0j},\cdots,r_{2^{2^{l-1}}-1,j}):=$ coefficients of $r_j$;\\
 3&end for;\\
4& for $i=0$ to $2^{2^{l-1}}-1$;\\
5&$v_i:=(r_{i0},\cdots,r_{i,2^{m-2^{l-1}}-1})$;\\
6&$t_i:={\rm IA}(v_i,b,m-2^{l-1},L)$;\\
 7&end for;\\
8&$t:=(t_0,\cdots,t_{2^{2^{l-1}}-1})$;\\
9&$f:=\sum_{i=0}^{2^{2^{l-1}}-1}t_i(x^{2^{2^{l-1}}}-x)x^i$;\\
&Return $f$;\\
\hline
\end{tabular}
\vskip0.2in
\begin{theorem}\label{interpolation}
The algorithm ${\rm IA}(v,a,m,L)$ costs no more than $2^{m-1}m$ multiplications, and no more than $2^mm(1+\log_2 m)$ additions.
\end{theorem}
\begin{proof}
This can be proved as Theorems \ref{addition} and \ref{multiplication}.
\end{proof}
\section{Multiplication}
In this section we use the evaluation and interpolation algorithms to give a multiplication algorithm ${\rm MA}(f,g,m,L)$.

The task of ${\rm MA}(f,g,m,L)$ is given in the following table.
\vskip0.2in
\begin{tabular}{ll}
  &Task of ${\rm MA}(f,g,m,L)$\\
\hline
 input 1  & $L$: a number;\\
 input 2  & $m$: a number $\leq 2^L$;\\
input 3&$(f,g)$: a pair of polynomials of degree $<2^m$ over ${\rm GF}(2^{2^L})$;\\
   install &$(\beta_1,\cdots,\beta_{2^L})$: a cantor basis of ${\rm GL}(2^{2^L})$;\\
   output& $fg$;\\
  \hline
\end{tabular}
\vskip0.2in
We now describe the mechanism of the multiplication algorithm.
\begin{lemma}[multiplication mechanism]
Let $L,m$ be numbers with  $1<m+1\leq 2^L$.
Let $f,g$ be polynomials of degree $<2^m$ over ${\rm GF}(2^{2^L})$, and $h$ the polynomial of degree $<2^{m+1}$ such that
$$h(\varpi_i)=f(\varpi_i)f(\varpi_i),\ i=0,\cdots,2^{m+1}-1.$$
Then $h(x)=f(x)g(x)$.
\end{lemma}
\begin{proof}
Obvious.
\end{proof}
The above lemma justifies the following pseudo-codes.
\vskip 0.2in
\begin{tabular}{ll}
 & Pseudo-codes of ${\rm MA}(f,g,m,L)$\\
\hline
1& $(u_0,u_1,\cdots,u_{2^{m+1}-1}):={\rm EA}(f,0,m+1,L)$;\\
2& $(v_0,v_1,\cdots,v_{2^{m+1}-1}):={\rm EA}(g,0,m+1,L)$;\\
3& for $i=0$ to $2^{m+1}-1$;\\
4&$w_i:=u_iv_i$;\\
5& end for\\
6&$w:=(w_0,\cdots, w_{2^{m+1}-1})$;\\
7&$h:={\rm IA}(w,0,m+1,L)$;\\
&Return $h$;\\
\hline
\end{tabular}
\vskip0.2in
\begin{theorem}\label{multiplicaion}
The algorithm ${\rm MA}(f,g,m,L)$ costs no more than $2^m(3m+5)$ multiplications  and $3\cdot2^{m+1}(m+1)(1+\log_2(m+1))$ additions in ${\rm GF}(2^{2^L})$.
\end{theorem}
\begin{proof}
Firstly, by Theorems \ref{addition} and \ref{multiplication},
Lines 1-2 each costs at most $2^m(m+1)$ multiplications and $2^{m+1}(m+1)(1+\log_2(m+1)$ additions.
Secondly by Theorem \ref{interpolation},
Line 7 costs  at most $2^m(m+1)$ multiplications  and $2^{m+1}(m+1)(1+\log_2(m+1)$ additions.
Thirdly, Line 4 costs at most  $2^{m+1}$ multiplications.
Therefore the algorithm totally costs
 at most $2^m(3m+5)$ multiplications  and $3\cdot2^{m+1}(m+1)(1+\log_2(m+1))$ additions.
The theorem is proved.
\end{proof}
\section{Special Multiplication}
In this section we  give a special multiplication algorithm  ${\rm SMA}(f,g,m)$.

The task of  ${\rm SMA}(f,g,m)$ is given in the following table.
\vskip0.2in
\begin{tabular}{ll}
  &Task of ${\rm SMA}(f,g,m)$\\
\hline
 input 1  & $m$: a number;\\
input 2&$(f,g)$: a pair of polynomials of degree $<2^m$ over ${\rm GF}(2)$;\\
   install 1& $L:=\lceil\log_2(m+1)\rceil$;\\
install 2 &$(\beta_1,\cdots,\beta_{2^L})$: a cantor basis of ${\rm GL}(2^{2^L})$;\\
   output& $fg$;\\
  \hline
\end{tabular}
\vskip0.2in
\begin{lemma}[special multiplication mechanism]
Let $f,g$ be polynomials of degree $<2^m$ over ${\rm GF}(2)$.
Write
$$f(x)=\sum_{i=0}^{2^m-1}a_ix^i, \
g(x)=\sum_{i=0}^{2^m-1}b_ix^i.$$
For $i=0,\cdots,2^{m-L+1}-1$, write $A_i=F_i(\beta_{2^L})$, and $B_i=G_i(\beta_{2^L})$, where $$F_i(x)=\sum_{j=0}^{2^{L-1}-1}a_{2^{L-1}i+j}x^j,\
G_i(x):=\sum_{j=0}^{2^{L-1}-1}b_{2^{L-1}i+j}x^j.$$
Let $H(x)=F(x)G(x)$, where
$$F(x)=\sum_{i=0}^{2^{m-L+1}-1}A_ix^i,\
G(x)=\sum_{i=0}^{2^{m-L+1}-1}B_ix^i.$$
Write
$H(x)=\sum_{i=0}^{2^{m-L+2}-1}C_ix^i$.
For $i=0,\cdots,2^{m-L+2}-1$, write $$C_i:=\sum_{j=0}^{2^L-1}h_{ij}\beta_{2^L}^j,\ h_{ij}=0,1.$$
Then
$$f(x)g(x)=\sum_{i=0}^{2^{m-L+2}-1}x^{2^{L-1}i}
\sum_{j=0}^{2^L-1}h_{ij}x^j.$$
\end{lemma}
\begin{proof}
Let $i=0,\cdots,2^{m-L+2}-1$, and write
$$H_i(x)=\sum_{j=0}^{2^L-1}h_{ij}x^j.$$
Then
\begin{eqnarray*}
  H_i(\beta_{2^L}) &=& C_i \\
   &=& \sum_{j=0}^iA_jB_{j-i}\\
&=&\sum_{j=0}^iF_j(\beta_{2^L})G_{j-i}(\beta_{2^L}).
\end{eqnarray*}
As $H_i(x)-\sum_{j=0}^iF_i(x)G_{j-i}(x)$ is of degree $<2^L$,
we conclude that
$$H_i(x)=\sum_{j=0}^iF_i(x)G_{j-i}(x).$$
It follows that
\begin{eqnarray*}
  h(x)&=&\sum_{i=0}^{2^{m-L+2}-1}x^{2^{L-1}i}H_i(x)\\
   &=&\sum_{i=0}^{2^{m-L+2}-1}x^{2^{L-1}i}\sum_{j=0}^iF_j(x)G_{j-i}(x)\\
   &=&\left(\sum_{j=0}^{2^{m-L+1}-1}F_j(x)x^{2^{L-1}j}\right)
\left(\sum_{j=0}^{2^{m-L+1}-1}G_j(x)x^{2^{L-1}j}\right)\\
&=&f(x)g(x).
\end{eqnarray*}
The lemma is proved.
\end{proof}
The above lemma justifies the following pseudo-codes.
\vskip0.2in
\begin{tabular}{ll}
 & Pseudo-codes of ${\rm SMA}(f,g,m)$\\
\hline
1& if $m=0$;\\
2&$h:=f(0)g(0)$;\\
3&else\\
4&$(a_0,\cdots,a_{2^m-1}):=$ coefficients of $f$;\\
5&$(b_0,\cdots,b_{2^m-1}):=$ coefficients of $g$;\\
6& for $i=0$ to $2^{m-L+1}-1$;\\
7& $A_i:=\sum_{j=0}^{2^{L-1}-1}a_{2^{L-1}i+j}\beta_{2^L}^j$;\\
8& $B_i:=\sum_{j=0}^{2^{L-1}-1}b_{2^{L-1}i+j}\beta_{2^L}^j$;\\
9&end for\\
10&$F(x):=\sum_{i=0}^{2^{m-L+1}-1}A_ix^i$;\\
11&$G(x):=\sum_{i=0}^{2^{m-L+1}-1}B_ix^i$;\\
12&$H(x):={\rm MA}(F,G,m-L+1)$;\\
13&$(C_0,\cdots,C_{2^{m-L+2}}):=$ coefficients of $H$;\\
14& for $i=0$ to $2^{m-L+2}-1$;\\
15& $C_i:=\sum_{j=0}^{2^L-1}h_{ij}\beta_{2^L}^j,\ h_{ij}=0,1$;\\
16&end for\\
17& $h(x):=\sum_{i=0}^{2^{m-L+2}-1}x^{2^{L-1}i}
\sum_{j=0}^{2^L-1}h_{ij}x^j$;\\
18&end if\\
&Return $h$;\\
\hline
\end{tabular}
\vskip0.2in

\begin{lemma}[multiplication cost]
The algorithm ${\rm SMA}(f,g,m)$ costs no more than
$2^{m-L+1}(3m+9)$ multiplications  and $2^{m-L+2}(3m+7)(1+\log_2(m+2))$  additions in ${\rm GF}(2^{2^L})$.
\end{lemma}
\begin{proof}
Firstly, Line 13, by  Theorem \ref{multiplication}, costs
 at most $2^{m-L+1}(3m+9)$ multiplications  and $2^{m-L+2}(3m+6)(1+\log_2(m+2))$ additions in ${\rm GF}(2^{2^L})$.
Secondly, Line 17 costs $2^{m+2}$ additions in ${\rm GF}(2^{2^L})$.
Therefore, the algorithm costs
 at most $2^{m-L+1}(3m+9)$ multiplications  and $2^{m-L+2}(3m+7)(1+\log_2(m+2))$ additions in ${\rm GF}(2^{2^L})$.
The lemma is proved.
\end{proof}
\begin{lemma}[coarse cost in bit operations]
The algorithm ${\rm SMA}(f,g,m)$ costs no more than  $$2^{m+2}(3m+9)(4m+3+\log_2(m+2))$$ bit operations.
\end{lemma}
\begin{proof}
By the last lemma,
 the algorithm costs
 at most $2^{m-L+1}(3m+9)$ multiplications  and $2^{m-L+2}(3m+7)(1+\log_2(m+2))$ additions in ${\rm GF}(2^{2^L})$.
Notice that, each addition in ${\rm GF}(2^{2^L})$ costs $2^L$ bit operations.
So the $2^{m-L+2}(3m+7)(1+\log_2(m+2))$ additions in ${\rm GF}(2^{2^L})$ contribute at most $2^{m+2}(3m+7)(1+\log_2(m+2)$ bit operations.
Also notice that each multiplication costs  at most $2^{2L+2}$ bit operations.
So the $2^{m-L+1}(3m+9)$ multiplications  in ${\rm GF}(2^{2^L})$ contribute at most $2^{m+L+3}(3m+9)$ bit operations.
It follows that
the algorithm  costs at most  $2^{m+2}(3m+9)(2^{L+1}+1+\log_2(m+2))$ bit operations.
The lemma now follows from the fact that $2^L\leq 2m+1$.
\end{proof}
\begin{theorem}
The algorithm ${\rm SMA}(f,g,m)$ costs $O(2^mm(\log m)^2)$ bit operations.
\end{theorem}
\begin{proof}
By the multiplication cost,
 the algorithm costs
 at most $2^{m-L+1}(3m+9)$ multiplications  and $2^{m-L+2}(3m+7)(1+\log_2(m+2)))$ additions in ${\rm GF}(2^{2^L})$.
By the proof of the last lemma, the $2^{m-L+2}(3m+7)(1+\log_2(m+2))$ additions in ${\rm GF}(2^{2^L})$ contribute  $O(2^mm\log_2m$ bit operations.
Recursively applying our special multiplication algorithm to carry multiplications in
${\rm GF}(2^{2^L})$, by the last lemma, each multiplication costs  at most $O(2^LL^2)$  bit operations.
So the $2^{m-L+1}(3m+9)$ multiplications contribute at most $O(2^mm(\log m)^2)$ bit operations.
The theorem is proved.
\end{proof}
The above theorem implies Theorem \ref{main}

\end{document}